\documentclass[12pt,fleqn]{article}

\usepackage{amsmath,amssymb,amsthm}
\usepackage[margin=1in]{geometry}
\usepackage[pdfpagemode=UseNone,pdfstartview=FitH,hypertex]{hyperref}

\newcommand{\bgset}[1]{\big\{#1\big\}}

\newcommand{\norm}[2][]{\left\|#2\right\|_{#1}}
\newcommand{\PS}[1]{$(\text{PS})_{#1}$}
\newcommand{\pnorm}[2][]{\if #1'' \left|#2\right|_p \else \left|#2\right|_{#1} \fi}
\newcommand{\pow}{{(p-2)/p}}
\newcommand{\QED}{\qedhere}
\newcommand{\restr}[2]{\left.#1\right|_{#2}}
\newcommand{\set}[1]{\left\{#1\right\}}
\newcommand{\wto}{\rightharpoonup}
\newcommand{\wop}{{p/(p-2)}}

\newcommand{\R}{\mathbb R}
\newcommand{\M}{{\mathcal M}}
\newcommand{\CN}{{\mathcal N}}

\DeclareMathOperator{\dist}{dist}
\DeclareMathOperator{\spn}{span}

\newtheorem{corollary}{Corollary}[section]
\newtheorem{lemma}[corollary]{Lemma}
\newtheorem{theorem}[corollary]{Theorem}

\numberwithin{equation}{section}

\title{\bf A nodal solution of the scalar field equation at the second minimax level}
\author{\bf Kanishka Perera\\
Department of Mathematical Sciences\\
Florida Institute of Technology\\
Melbourne, FL 32901, USA\\
\em kperera@fit.edu\\
[\bigskipamount]
\bf Cyril Tintarev\\
Department of Mathematics\\
Uppsala University\\
75 106, Uppsala, Sweden\\
\em tintarev@math.uu.se}
\date{}

\begin{document}

\maketitle

\begin{abstract}
We prove the existence of a sign-changing eigenfunction at the second minimax level of the eigenvalue problem for the scalar field equation under a slow decay condition on the potential near infinity. The proof involves constructing a set consisting of sign\nobreakdash-changing functions that is dual to the second minimax class. We also obtain a nonradial sign-changing eigenfunction at this level when the potential is radial.
\end{abstract}

\section{Introduction}

Consider the eigenvalue problem for the scalar field equation
\begin{equation} \label{1.1}
- \Delta u + V(x)\, u = \lambda\, |u|^{p-2}\, u, \quad u \in H^1(\R^N),
\end{equation}
where $N \ge 2$, $V \in L^\infty(\R^N)$ satisfies
\begin{gather}
V(x) \ge 0 \quad \forall x \in \R^N, \label{1.2}\\[10pt]
\lim_{|x| \to \infty} V(x) = V^\infty > 0, \label{1.3}
\end{gather}
and $p \in (2,2^\ast)$. Here $2^\ast = 2N/(N-2)$ if $N \ge 3$ and $2^\ast = \infty$ if $N = 2$. Let
\[
I(u) = \int_{\R^N} |u|^p, \quad J(u) = \int_{\R^N} |\nabla u|^2 + V(x)\, u^2, \quad u \in H^1(\R^N).
\]
Then the eigenfunctions of \eqref{1.1} on the manifold
\[
\M = \bgset{u \in H^1(\R^N) : I(u) = 1}
\]
and the corresponding eigenvalues coincide with the critical points and the corresponding critical values of the constrained functional $\restr{J}{\M}$, respectively. Equation \eqref{1.1} has been studied extensively for more than three decades (see Bahri and Lions \cite{MR1450954} for a detailed account). The main difficulty here is the lack of compactness inherent in this problem. This lack of compactness originates from the invariance of $\R^N$ under the action of the noncompact group of translations, and manifests itself in the noncompactness of the Sobolev imbedding $H^1(\R^N) \hookrightarrow L^p(\R^N)$. This in turn implies that the manifold $\M$ is not weakly closed in $H^1(\R^N)$ and that $\restr{J}{\M}$ does not satisfy the usual Palais-Smale compactness condition at all energy levels.

Least energy solutions, also called ground states, are well-understood. In general, the infimum
\[
\lambda_1 := \inf_{u \in \M}\, J(u)
\]
is not attained. For the autonomous problem at infinity,
\[
- \Delta u + V^\infty\, u = \lambda\, |u|^{p-2}\, u, \quad u \in H^1(\R^N),
\]
the corresponding functional
\[
J^\infty(u) = \int_{\R^N} |\nabla u|^2 + V^\infty\, u^2, \quad u \in H^1(\R^N)
\]
attains its infimum
\[
\lambda^\infty_1 := \inf_{u \in \M}\, J^\infty(u) > 0
\]
at a radial function $w^\infty_1 > 0$ (see Berestycki and Lions \cite{MR695535}), and this minimizer is unique up to translations (see Kwong \cite{MR969899}). For the nonautonomous problem, we have $\lambda_1 \le \lambda^\infty_1$ by \eqref{1.3} and the translation invariance of $J^\infty$, and $\lambda_1$ is attained if the inequality is strict (see Lions \cite{MR778970,MR778974}).

As for higher energy solutions, also called bound states, radial solutions have been extensively studied when the potential $V$ is radially symmetric (see, e.g., Berestycki and Lions \cite{MR695536}, Jones and K{\"u}pper \cite{MR846391}, Grillakis \cite{MR1054554}, Bartsch and Willem \cite{MR1237913}, and Conti et al. \cite{MR1773818}). The subspace $H^1_r(\R^N)$ of $H^1(\R^N)$ consisting of radially symmetric functions is compactly imbedded into $L^p(\R^N)$ for $p \in (2,2^\ast)$ by a compactness result of Strauss \cite{MR0454365}. So in this case the restrictions of $J$ and $J^\infty$ to $\M \cap H^1_r(\R^N)$ have increasing and unbounded sequences of critical values given by a standard minimax scheme. Furthermore, Sobolev imbeddings remain compact for subspaces with any sufficiently robust symmetry (see, e.g., Bartsch and Willem \cite{MR1244943}, Bartsch and Wang \cite{MR1349229}, and Devillanova and Solimini \cite{MR2895950}). As for multiplicity in the nonsymmetric case, Zhu \cite{MR969507}, Hirano \cite{MR1866736}, Clapp and Weth \cite{MR2103845}, and Perera \cite{Pe17} have given sufficient conditions for the existence of $2$, $3$, $N/2 + 1$, and $N - 1$ pairs of solutions, respectively (see also Li \cite{MR1205151}). There is also an extensive literature on multiple solutions of scalar field equations in topologically nontrivial unbounded domains (see the survey paper of Cerami \cite{MR2278729}).

The second minimax level is defined as
\[
\lambda_2 := \inf_{\gamma \in \Gamma_2}\, \max_{u \in \gamma(S^1)}\, J(u),
\]
where $\Gamma_2$ is the class of all odd continuous maps $\gamma : S^1 \to \M$ and $S^1$ is the unit circle. In the present paper we address the question of whether solutions of \eqref{1.1} at this level, previously obtained in Perera and Tintarev \cite{PeTi}, are nodal (sign-changing), as it would be expected in the linearized problem where the nonlinearity $|u|^{p-2}\, u$ is replaced by $U(x)\, v$ with the decaying potential $U(x) = |u|^{p-2}$. Ultimately, nodality of a solution $u$ follows from the orthogonality relation \eqref{orth} below, which can be understood as $\left<u^\ast,v_1(u)\right> = 0$, where $v_1(u)$ is the ground state of the linearized problem with the potential $U(x) = |u|^{p-2}$ and $u^\ast = |u|^{p-2}\, u$ is the duality conjugate of $u$. A similar argument has been used in Tarantello \cite{MR1141725} to obtain a nodal solution of a critical exponent problem in a bounded domain (see also Coffman \cite{MR940607}).

Recall that
\[
w_1^\infty(x) \sim C_0\, \frac{e^{- \sqrt{V^\infty}\, |x|}}{|x|^{(N-1)/2}} \text{ as } |x| \to \infty
\]
for some constant $C_0 > 0$, and that there are constants $0 < a_0 \le \sqrt{V^\infty}$ and $C > 0$ such that if $\lambda_1$ is attained at $w_1 \ge 0$, then
\[
w_1(x) \le C\, e^{- a_0\, |x|} \quad \forall x \in \R^N
\]
(see Gidas et al. \cite{MR544879}). Write
\[
V(x) = V^\infty - W(x),
\]
so that $W(x) \to 0$ as $|x| \to \infty$ by \eqref{1.3}, and write $\pnorm[q]{\cdot}$ for the norm in $L^q(\R^N)$. Our main result is the following.

\begin{theorem} \label{Theorem 1.1}
Assume that $V \in L^\infty(\R^N)$ satisfies \eqref{1.2} and \eqref{1.3}, $p \in (2,2^\ast)$, and
\[
W(x) \ge c_0\, e^{- a\, |x|} \quad \forall x \in \R^N
\]
for some constants $0 < a < a_0$ and $c_0 > 0$. If $W \in L^\wop(\R^N)$ and
\[
\pnorm[\wop]{W} < (2^\pow - 1)\, \lambda^\infty_1,
\]
then equation \eqref{1.1} has a nodal solution on $\M$ for $\lambda = \lambda_2$.
\end{theorem}

Nodal solutions to a closely related problem have been obtained in Zhu \cite{MR969507} and Hirano \cite{MR1866736} under assumptions different from those in Theorem \ref{Theorem 1.1}. We will prove this theorem by constructing a set $F \subset \M$ consisting of sign-changing functions that is dual to the class $\Gamma_2$.

As a corollary we obtain a nonradial nodal solution of \eqref{1.1} at the level $\lambda_2$ when $V$ is radial. Let $\Gamma_{2,\, r}$ denote the class of all odd continuous maps from $S^1$ to $\M_r = \M \cap H^1_r(\R^N)$ and set
\[
\lambda_{2,\, r} := \inf_{\gamma \in \Gamma_{2,\, r}}\, \max_{u \in \gamma(S^1)}\, J(u), \quad \lambda^\infty_{2,\, r} := \inf_{\gamma \in \Gamma_{2,\, r}}\, \max_{u \in \gamma(S^1)}\, J^\infty(u).
\]
Since the imbedding $H^1_r(\R^N) \hookrightarrow L^p(\R^N)$ is compact by the result of Strauss \cite{MR0454365}, these levels are critical for $\restr{J}{\M_r}$ and $\restr{J^\infty}{\M_r}$, respectively. Since $\Gamma_{2,\, r} \subset \Gamma_2$ and
\[
\lambda^\infty_2 := \inf_{\gamma \in \Gamma_2}\, \max_{u \in \gamma(S^1)}\, J^\infty(u)
\]
is not critical for $\restr{J^\infty}{\M}$ (see, e.g., Cerami \cite{MR2278729}), we have $\lambda_2 \le \lambda_{2,\, r}$ and $\lambda^\infty_2 < \lambda^\infty_{2,\, r}$. It was shown in Perera and Tintarev \cite[Theorem 1.3]{PeTi} that if $W \in L^\wop(\R^N)$ and
\[
\pnorm[\wop]{W} < \lambda^\infty_{2,\, r} - \lambda^\infty_2,
\]
then $\lambda_2 < \lambda_{2,\, r}$ and every nodal solution of \eqref{1.1} on $\M$ with $\lambda = \lambda_2$ is nonradial. Combining Theorem \ref{Theorem 1.1} with this result now gives us the following corollary.

\begin{corollary}
Assume that $V \in L^\infty(\R^N)$ is radial and satisfies \eqref{1.2} and \eqref{1.3}, $p \in (2,2^\ast)$, and
\[
W(x) \ge c_0\, e^{- a\, |x|} \quad \forall x \in \R^N
\]
for some constants $0 < a < a_0$ and $c_0 > 0$. If $W \in L^\wop(\R^N)$ and
\[
\pnorm[\wop]{W} < \min \set{(2^\pow - 1)\, \lambda^\infty_1,\lambda^\infty_{2,\, r} - \lambda^\infty_2},
\]
then $\lambda_2 < \lambda_{2,\, r}$ and equation \eqref{1.1} has a nonradial nodal solution on $\M$ for $\lambda = \lambda_2$.
\end{corollary}

We will give the proof of Theorem \ref{Theorem 1.1} in the next section.

\section{Proof of Theorem \ref{Theorem 1.1}}

Let
\[
\norm{u} = \sqrt{J(u)}, \quad u \in H^1(\R^N).
\]

\begin{lemma}
$\norm{\cdot}$ is an equivalent norm on $H^1(\R^N)$.
\end{lemma}

\begin{proof}
Since $V \in L^\infty(\R^N)$, $\norm{u} \le C \norm[H^1(\R^N)]{u}$ for all $u \in H^1(\R^N)$ for some constant $C > 0$. If the reverse inequality does not hold for any $C > 0$, then there exists a sequence $u_k \in H^1(\R^N)$ such that
\begin{gather}
\pnorm[2]{u_k} = 1, \label{2.4}\\[10pt]
\pnorm[2]{\nabla u_k} \to 0, \label{2.5}\\[7.5pt]
\int_{\R^N} V(x)\, u_k^2 \to 0. \label{2.6}
\end{gather}
By \eqref{1.3}, there exists $R > 0$ such that
\[
V(x) \ge \frac{V^\infty}{2} > 0 \quad \forall x \in \R^N \setminus B_R,
\]
where $B_R = \set{x \in \R^N : |x| < R}$. Then $\pnorm[L^2(\R^N \setminus B_R)]{u_k} \to 0$ by \eqref{2.6} and \eqref{1.2}. By \eqref{2.4} and \eqref{2.5}, $u_k$ is bounded in $H^1(\R^N)$ and hence converges weakly in $H^1(\R^N)$ to some $w$ for a renamed subsequence. By the weak lower semicontinuity of the gradient seminorm, then $\pnorm[2]{\nabla w} = 0$, so $w$ is a constant function. This constant is necessarily zero since $w \in H^1(\R^N)$. Consequently, $\pnorm[L^2(B_R)]{u_k} \to 0$ by the compactness of the imbedding $H^1(\R^N) \hookrightarrow L^2(B_R)$. Thus, $\pnorm[2]{u_k} \to 0$, contradicting \eqref{2.4}.
\end{proof}

For $u \in H^1(\R^N)$, let
\[
K_u(v) = \int_{\R^N} |u(x)|^{p-2}\, v^2, \quad v \in H^1(\R^N).
\]

\begin{lemma} \label{Lemma 1}
The map $H^1(\R^N) \times H^1(\R^N) \to \R,\, (u,v) \mapsto K_u(v)$ is continuous with respect to norm convergence in $u$ and weak convergence in $v$, i.e., $K_{u_k}(v_k) \to K_u(v)$ whenever $u_k \to u$ in $H^1(\R^N)$ and $v_k \wto v$ in $H^1(\R^N)$.
\end{lemma}

\begin{proof}
It suffices to show that $K_{u_k}(v_k) \to K_u(v)$ for a renamed subsequence of $(u_k,v_k)$. By the continuity of the Sobolev imbedding $H^1(\R^N) \hookrightarrow L^p(\R^N)$, $u_k \to u$ in $L^p(\R^N)$ and $v_k$ is bounded in $L^p(\R^N)$. Then
\[
\int_{\R^N \setminus B_R} |u_k(x)|^{p-2}\, v_k^2 + |u(x)|^{p-2}\, v^2 \le \pnorm[L^p(\R^N \setminus B_R)]{u_k}^{p-2} \pnorm[p]{v_k}^2 + \pnorm[L^p(\R^N \setminus B_R)]{u}^{p-2} \pnorm[p]{v}^2
\]
by the H\"{o}lder inequality and the right-hand side can be made arbitrarily small by taking $R > 0$ and $k$ sufficiently large, where $B_R = \set{x \in \R^N : |x| < R}$. By the compactness of the imbedding $H^1(B_R) \hookrightarrow L^p(B_R)$, $(u_k,v_k) \to (u,v)$ strongly in $L^p(B_R) \times L^p(B_R)$ and a.e.\! in $B_R \times B_R$ for a renamed subsequence. Then
\[
\int_{B_R} |u_k(x)|^{p-2}\, v_k^2 \to \int_{B_R} |u(x)|^{p-2}\, v^2
\]
by the elementary inequality
\[
|a|^{p-2}\, b^2 \le \left(1 - \frac{2}{p}\right) |a|^p + \frac{2}{p}\, |b|^p \quad \forall a, b \in \R
\]
and the dominated convergence theorem. Thus, the conclusion follows.
\end{proof}

For $u \in H^1(\R^N) \setminus \set{0}$, let
\[
\M_u = \bgset{v \in H^1(\R^N) : K_u(v) = 1}.
\]

\begin{lemma} \label{Lemma 2.1}
For $u \in H^1(\R^N) \setminus \set{0}$, the infimum
\begin{equation} \label{mu_1}
\mu_1(u) := \inf_{v \in \M_u}\, J(v)
\end{equation}
is attained at a unique $v_1(u) > 0$, and the even map $H^1(\R^N) \setminus \set{0} \to H^1(\R^N),\, u \mapsto v_1(u)$ is continuous.
\end{lemma}

\begin{proof}
The functional $K_u$ is weakly continuous on $H^1(\R^N)$ by Lemma \ref{Lemma 1} and $J$ is weakly lower semicontinuous, so the infimum in \eqref{mu_1} is attained at some $v_1(u) \in \M_u$. By the strong maximum principle, $v_1(u) > 0$. Then the right-hand side of the well-known Jacobi identity
\[
J(v) - \mu_1(u)\, K_u(v) = \int_{\R^N} v_1(u)^2\, \left|\nabla \left(\frac{v}{v_1(u)}\right)\right|^2
\]
vanishes at $v \in \M_u$ if and only if $v = v_1(u)$, so the minimizer is unique.

Let $u_k \to u \ne 0$ in $H^1(\R^N)$. By Lemma \ref{Lemma 1}, $K_{u_k}(v_1(u)) \to K_u(v_1(u)) = 1$, so for sufficiently large $k$, $K_{u_k}(v_1(u)) > 0$ and $v_1(u)/\sqrt{K_{u_k}(v_1(u))} \in \M_{u_k}$. Then
\[
J(v_1(u_k)) \le J\left(\frac{v_1(u)}{\sqrt{K_{u_k}(v_1(u))}}\right) = \frac{\mu_1(u)}{K_{u_k}(v_1(u))} \to \mu_1(u),
\]
so
\begin{equation} \label{1}
\limsup J(v_1(u_k)) \le \mu_1(u).
\end{equation}
In particular, $v_1(u_k)$ is bounded in $H^1(\R^N)$ and hence converges weakly in $H^1(\R^N)$ to some $v$ for a renamed subsequence. Then $1 = K_{u_k}(v_1(u_k)) \to K_u(v)$ by Lemma \ref{Lemma 1}, so $K_u(v) = 1$ and hence $v \in \M_u$. Since $J$ is weakly lower semicontinuous, then
\begin{equation} \label{2}
\mu_1(u) \le J(v) \le \liminf J(v_1(u_k)).
\end{equation}
Combining \eqref{1} and \eqref{2} gives $\lim J(v_1(u_k)) = J(v) = \mu_1(u)$, so $\norm{v_1(u_k)} \to \norm{v}$ and $v = v_1(u)$ by the uniqueness of the minimizer. Since $v_1(u_k) \wto v$, then $v_1(u_k) \to v_1(u)$.
\end{proof}

For $u \in H^1(\R^N) \setminus \set{0}$, let
\[
L_u(v) = \int_{\R^N} |u(x)|^{p-2}\, v_1(u)\, v, \quad v \in H^1(\R^N)
\]
and let
\[
\CN_u = \bgset{v \in \M_u : L_u(v) = 0}.
\]

\begin{lemma}
For $u \in H^1(\R^N) \setminus \set{0}$, the infimum
\begin{equation} \label{mu_2}
\mu_2(u) := \inf_{v \in \CN_u}\, J(v)
\end{equation}
is attained at some $v_2(u)$.
\end{lemma}

\begin{proof}
The functional $K_u$ is weakly continuous on $H^1(\R^N)$ by Lemma \ref{Lemma 1}, $L_u$ is a bounded linear functional on $H^1(\R^N)$, and $J$ is weakly lower semicontinuous, so the infimum in \eqref{mu_2} is attained at some $v_2(u) \in \CN_u$.
\end{proof}

\begin{lemma} \label{Lemma 2.3}
For $u \in H^1(\R^N) \setminus \set{0}$, $v_1(u)$ and $v_2(u)$ are linearly independent and
\[
J(v) \le \mu_2(u) \int_{\R^N} |u(x)|^{p-2}\, v^2 \quad \forall v \in \spn \set{v_1(u),v_2(u)}.
\]
\end{lemma}

\begin{proof}
Since $L_u(v_2) = 0$ and $K_u(v_1) = 1$, $v_1$ and $v_2$ are linearly independent. We have
\[
\int_{\R^N} \nabla v_1 \cdot \nabla w + V(x)\, v_1 w = \mu_1(u) \int_{\R^N} |u(x)|^{p-2}\, v_1 w \quad \forall w \in H^1(\R^N),
\]
and testing with $v_2$ gives
\[
\int_{\R^N} \nabla v_1 \cdot \nabla v_2 + V(x)\, v_1 v_2 = \mu_1(u) \int_{\R^N} |u(x)|^{p-2}\, v_1 v_2 = 0.
\]
Then
\begin{multline*}
J(c_1 v_1 + c_2 v_2) = c_1^2\, J(v_1) + c_2^2\, J(v_2) = c_1^2\, \mu_1(u) + c_2^2\, \mu_2(u)\\[7.5pt]
\le (c^1_2 + c_2^2)\, \mu_2(u) = \mu_2(u) \int_{\R^N} |u(x)|^{p-2}\, (c_1 v_1 + c_2 v_2)^2. \QED
\end{multline*}
\end{proof}

Let
\[
h(u) = \int_{\R^N} |u|^{p-2}\, u\, v_1(u), \quad u \in H^1(\R^N) \setminus \set{0}
\]
and let
\[
F = \bgset{u \in \M : h(u) = 0}.
\]
By Lemma \ref{Lemma 2.1}, $h$ is continuous and hence $F$ is closed.

\begin{lemma} \label{Lemma 3.1}
$\gamma(S^1) \cap F \ne \emptyset$ for all $\gamma \in \Gamma_2$.
\end{lemma}

\begin{proof}
Since $h \circ \gamma : S^1 \to \R$ is an odd continuous map by Lemma \ref{Lemma 2.1}, $h(u) = 0$ for some $u \in \gamma(S^1)$ by the intermediate value theorem.
\end{proof}

\begin{lemma} \label{Lemma 3.2}
$J(u) \ge \lambda_2$ for all $u \in F$.
\end{lemma}

\begin{proof}
Since $K_u(u) = I(u) = 1$ and $L_u(u) = h(u) = 0$, $u \in \CN_u$ and hence
\begin{equation} \label{2.1}
J(u) \ge \mu_2(u).
\end{equation}
By Lemma \ref{Lemma 2.3}, $v_1(u)$ and $v_2(u)$ are linearly independent and hence we can define an odd continuous map $\gamma_u : S^1 \to \M \cap \spn \set{v_1(u),v_2(u)}$ by
\[
\gamma_u(e^{i \theta}) = \frac{v_1(u)\, \cos \theta + v_2(u)\, \sin \theta}{\pnorm{v_1(u)\, \cos \theta + v_2(u)\, \sin \theta}}, \quad \theta \in [0,2 \pi].
\]
Take $v_0 \in \gamma_u(S^1)$ such that
\begin{equation} \label{2.2}
J(v_0) = \max_{v \in \gamma_u(S^1)}\, J(v) \ge \lambda_2.
\end{equation}
By Lemma \ref{Lemma 2.3} and the H\"older inequality,
\begin{equation} \label{2.3}
J(v_0) \le \mu_2(u) \pnorm[p]{u}^{p-2} \pnorm[p]{v_0}^2 = \mu_2(u).
\end{equation}
Combining \eqref{2.1}--\eqref{2.3} gives $J(u) \ge \lambda_2$.
\end{proof}

It was shown in Perera and Tintarev \cite[Proposition 3.1]{PeTi} that
\[
\lambda^\infty_1 < \lambda_2 < \big(\lambda_1^\wop + (\lambda_1^\infty)^\wop\big)^\pow
\]
under the hypotheses of Theorem \ref{Theorem 1.1}, so $\restr{J}{\M}$ satisfies the \PS{} condition at the level $\lambda_2$ (see, e.g., Cerami \cite{MR2278729}). Let $K^{\lambda_2}$ denote the set of critical points of $\restr{J}{\M}$ at this level.

\begin{lemma} \label{Lemma 3.3}
$K^{\lambda_2} \cap F \ne \emptyset$
\end{lemma}

\begin{proof}
Suppose $K^{\lambda_2} \cap F = \emptyset$. Since $K^{\lambda_2}$ is compact by the \PS{} condition and $F$ is closed, there is a $\delta > 0$ such that the $\delta$-neighborhood $N_\delta(K^{\lambda_2}) = \set{u \in \M : \dist (u,K^{\lambda_2}) \le \delta}$ does not intersect $F$. By the first deformation lemma, there are $\varepsilon > 0$ and an odd continuous map $\eta : \M \to \M$ such that
\[
\eta(J^{\lambda_2 + \varepsilon}) \subset J^{\lambda_2 - \varepsilon} \cup N_\delta(K^{\lambda_2}),
\]
where $J^a = \set{u \in \M : J(u) \le a}$ for $a \in \R$ (see Corvellec et al. \cite{MR94c:58026}). Since $J > \lambda_2 - \varepsilon$ on $F$ by Lemma \ref{Lemma 3.2} and $N_\delta(K^{\lambda_2}) \cap F = \emptyset$, then $\eta(J^{\lambda_2 + \varepsilon}) \cap F = \emptyset$. Take $\gamma \in \Gamma_2$ such that $\gamma(S^1) \subset J^{\lambda_2 + \varepsilon}$ and let $\widetilde{\gamma} = \eta \circ \gamma$. Then $\widetilde{\gamma} \in \Gamma_2$ and $\widetilde{\gamma}(S^1) \cap F = \emptyset$, contrary to Lemma \ref{Lemma 3.1}.
\end{proof}

We are now ready to prove Theorem \ref{Theorem 1.1}.

\begin{proof}[Proof of Theorem \ref{Theorem 1.1}]
By Lemma \ref{Lemma 3.3}, there is some $u \in K^{\lambda_2} \cap F$, which then is a solution of \eqref{1.1} on $\M$ for $\lambda = \lambda_2$ satisfying
\begin{equation} \label{orth}
\int_{\R^N} |u|^{p-2}\, u\, v_1(u) = 0.
\end{equation}
Since $v_1(u) > 0$ by Lemma \ref{Lemma 2.1}, $u$ is nodal.
\end{proof}

\def\cdprime{$''$}

\end{document}